\documentclass{amsart}

\usepackage{amssymb,amsmath,amsthm}
\usepackage{mathrsfs}
\usepackage{enumerate} 
\usepackage[hmargin=3cm,vmargin=3cm]{geometry} 
\usepackage[all,cmtip]{xy} 
\usepackage{multirow}
\usepackage{bbm}
\usepackage{float}
\usepackage{longtable}
\usepackage{tikz}
\usepackage[numbers]{natbib}
\usepackage{subfigure}

\usetikzlibrary{matrix}
\usetikzlibrary{arrows,shapes}

\tikzstyle{white}=[circle,draw=black!100,fill=white!100,thick,inner sep=0pt,minimum size =2mm]
\tikzstyle{blank}=[circle,draw=white!100,fill=white!100,thick,inner sep=0pt,minimum size =2mm]

\usepackage{soul}

\newtheorem{Theorem}{Theorem}[section]
\theoremstyle{plain}
\newtheorem{Thm}[Theorem]{Theorem}

\newtheorem{Lem}[Theorem]{Lemma}

\newtheorem{Cor}[Theorem]{Corollary}

\theoremstyle{definition}
\newtheorem{Defn}[Theorem]{Definition}

\newtheorem{Exam}[Theorem]{Example}

\newtheorem{Rem}[Theorem]{Remark}

\makeatletter
\renewcommand*\env@matrix[1][*\c@MaxMatrixCols c]{%
  \hskip -\arraycolsep
  \let\@ifnextchar\new@ifnextchar
  \array{#1}}
\makeatother

\def\NN{\mathbb N}

\def\RR{\mathbb R}

\def\fR{\mathfrak R}
\def\fL{\mathfrak L}

\def\11{\mathbbm{1}}

\newcommand{\seq}[1]{\left\{#1\right\}}
\newcommand{\bra}[1]{\left(#1\right)}

\newcommand{\abs}[1]{\left|#1\right|}

\renewcommand{\epsilon}{\varepsilon}

\def\diam{\textup{diam}}

\tikzstyle{red}=[circle,draw=red!80,fill=red!80,thick,inner sep=0pt,minimum size =2mm]
\tikzstyle{blue}=[circle,draw=blue!80,fill=blue!80,thick,inner sep=0pt,minimum size =2mm]
\tikzstyle{white}=[circle,draw=black!100,fill=white!100,thick,inner sep=0pt,minimum size =2mm]
\tikzstyle{black}=[circle,draw=black!100,fill=black!100,thick,inner sep=0pt,minimum size =1.5mm]

\title{Additive Combination Spaces}
\author{Stephen S\'{a}nchez}

\begin{document}



\begin{abstract}
 We introduce a class of metric spaces called $p$-additive combinations and show that for such spaces we may deduce information about their $p$-negative type behaviour by focusing on a relatively small collection of almost disjoint metric subspaces, which we call the components. In particular we deduce a formula for the $p$-negative type gap of the space in terms of the $p$-negative type gaps of the components, independent of how the components are arranged in the ambient space. This generalizes earlier work on metric trees by Doust and Weston \cite{DW08a,DW08b}. The results hold for semi-metric spaces as well, as the triangle inequality is not used. 
\end{abstract}

\maketitle{}

\section{Introduction}

The notion of $p$-negative type is a non-linear property of metric spaces with strong connections to embedding theory. An early example of such a connection is Schoenberg's classical result that a metric space is isometric to a subset of a Euclidean space if and only if it has 2-negative type \cite{Sch37}. This was later generalized to $L_p$ spaces by Bretagnolle, Dacunha-Castelle and Krivine \cite{BDCK67}, who showed that for $0< p \leq 2$, a real normed space is linearly isometric to a linear subspace of some $L_p$ space if and only if it has $p$-negative type. While the $p$-negative type properties of a space $(X,d)$ are determined by the $p$-negative type properties of all of its finite subspaces, this is not always a fruitful method of inquiry due to the multitude of spaces to consider. While we may bound the supremal $p$-negative type of $(X,d)$ from above by looking at only a single subspace of $(X,d)$ the same cannot be said for bounding from below. In this article we detail a class of spaces, which 
we call $p$-additive combinations, for which we may determine lower bounds on the supremal $p$-negative type properties of a space by looking at only relatively few almost disjoint metric subspaces.

\begin{Defn}\label{negtype}
Let $(X,d)$ be a metric space and $p \geq 0$. Then:
\begin{enumerate}[(i)]
 \item $(X,d)$ has \emph{$p$-negative type} if and only if for all natural numbers $n \geq 2$, all finite subsets $\seq{x_1,\ldots,x_n} \subset X$, and all choices of real numbers $\alpha_1, \ldots, \alpha_n$ with $\alpha_1 + \cdots + \alpha_n = 0$, we have:
\begin{equation}\label{type}
\sum_{1 \leq i, j \leq n} d(x_i,x_j)^p \alpha_i \alpha_j \leq 0.
\end{equation}
 \item $(X,d)$ has \emph{strict $p$-negative type} if and only if it has $p$-negative type and the inequalities \eqref{type} are all strict except in the trivial case $\bra{\alpha_1,\ldots, \alpha_n} = \bra{0, \ldots, 0}$.
\end{enumerate}
\end{Defn}

It is well known that $p$-negative type possesses the following interval property: if a metric space $\bra{X,d}$ has $p$-negative type, then it has $q$-negative type for all $0 \leq q \leq p$ (see \cite[p.~11]{WW75}). So it is sensible to define the following.

\begin{Defn}\label{supremal}
The \emph{supremal $p$-negative type} $\wp \bra{X,d}$ of a metric space $\bra{X,d}$ is
\[
\wp \bra{X,d} = \sup \seq{p : (X,d) \text{ has } p\text{-negative type}}.
\]
\end{Defn}
If $\wp(X,d)$ is finite then it is easy to see that $(X,d)$ does actually have $\wp(X,d)$-negative type. We write $\wp(X)$, or simply $\wp$, if the metric space is clear from context. 

Calculating $\wp$ for a general metric space is a difficult non-linear problem. Recent work by S\'{a}nchez \cite{San12}, using results of Wolf \cite{Wol12} and Li and Weston \cite{LW10}, gives a method of calculating, at least numerically, $\wp(X,d)$ for a given finite metric space $(X,d)$. However, it struggles with spaces of many points and requires us to work with one space at a time. So it is interesting to look at bounding $\wp$ from above or below for a collection of finite spaces, and indeed for many infinite spaces this seems to be the best that we can hope for.

A method for finding upper bounds on $\wp(X,d)$ comes straight from the definition: if we find a collection of points $\seq{x_1, \ldots, x_n} \subset X$ and numbers $\alpha_1, \ldots, \alpha_n$ with $\alpha_1 + \cdots + \alpha_n =0$ for which condition \eqref{type} fails to hold for some exponent $q$, then we conclude that $\wp < q$. More generally, if $(Y,\delta)$ can be isometrically embedded in $(X,d)$ then we have $\wp(X,d) \leq \wp(Y,\delta)$.

Lower bounds on $\wp(X,d)$ are far more difficult to obtain. As just noted above, if we can embed $(X,d)$ into some other space $(Z,d')$, then we know that $\wp(X,d) \geq \wp(Z,d')$. This is of limited use since bounding the value of $\wp(Z,'d)$ may be an even more complicated problem. A different method of bounding $\wp(X,d)$ from below makes us of the \emph{$p$-negative type gap} of $(X,d)$, first introduced in \cite{DW08a,DW08b}. This numerical quantity (defined below) measures how strictly $(X,d)$ has strict $p$-negative type. If non-zero, this may be used, along with some other properties of $(X,d)$, to bound $\wp(X,d)$ from below, see for instance \cite[Theorem~5.1]{DW08a,DW08b} and \cite[Theorem~3.3]{LW10}. We will use such a bound in Section~\ref{boundsection}.

\begin{Defn}\label{typegap}
Let $(X,d)$ be a metric space with strict $p$-negative type. The \emph{$p$-negative type gap $\Gamma^p_X$} is the largest non-negative constant $\Gamma$ such that
\[
 \frac{\Gamma}{2} \bra{ \sum_{l=1}^n \abs{\alpha_l} }^2 + \sum_{1 \leq i, j \leq n} d(x_i,x_j)^p \alpha_i \alpha_j \leq 0
\]
for all natural numbers $n \geq 2$, all finite subsets $\seq{x_1, \ldots, x_n} \subseteq X$ and all choices of real numbers $\alpha_1, \ldots, \alpha_n$ with $\alpha_1 + \cdots + \alpha_n =0$.
\end{Defn}

The definition of $\Gamma^p_X$ given above is not the original form in which it was defined in \cite{DW08a,DW08b}, and its translation into the $p$-negative type setting gives the awkward scaling factor above. We will in fact work with its original incarnation, which we come to in Section~\ref{GeneralizedRoundness}.

The main result of this paper is the following theorem. It shows that if $(X,d)$ is what we call a $p$-additive combination space, then we can deduce information about the $p$-negative type, strict $p$-negative type and $p$-negative type gap properties of $(X,d)$ simply by looking at a relatively small collection of metric subspaces within $(X,d)$.

\begin{Thm}\label{main}
Suppose $p \geq 0$. Let $(X,d)$ be a $p$-additive combination of $(X_1,d_1),\ldots, (X_n,d_n)$. 
\begin{enumerate}[(i)]
 \item If $(X_1,d_1),\ldots,(X_n,d_n)$ all have $p$-negative type, then so does $(X,d)$.
 \item If $(X_1,d_1),\ldots,(X_n,d_n)$ all have strict $p$-negative type, then so does $(X,d)$.
 \item If $\Gamma^p_{X_1},\ldots,\Gamma^p_{X_n} > 0$, then $\Gamma^p_X >0$ and is given by
\[
\Gamma^p_X = \bra{ \sum_{i=1}^n \bra{\Gamma^p_{X_i}}^{-1} }^{-1}.
\]
\end{enumerate}
\end{Thm}

We will formally define $p$-additive combinations in the coming sections. It happens that the above theorem can be deduced quite easily once the $p=1$ case is established. For this reason we shall first focus on the $p=1$ case and additive combinations, extending to other values of $p$ later on. Quite interestingly, Theorem~\ref{main} holds independently of how the spaces $(X_1,d_1), \ldots, (X_n,d_n)$ are combined to form $(X,d)$. 

\begin{Rem}\label{matrix}
 We note briefly that Kokkendorff proved the parts (i) and (ii) of Theorem~\ref{main} for the case $p=1$ in his Ph.D.~thesis \cite[Ch 4, Cor 5]{Kok02}. His result spoke in terms of \emph{one point unions} and gave an algebraic proof. We feel that our exposition in terms of generalized roundness $p$ gives a more geometric understanding of the result, and allows us to extend to part (iii) more naturally. 
\end{Rem}

\section{Additive Combination Spaces}

In essence, an additive combination of $(X_1,d_1)$ and $(X_2,d_2)$ is a space made by picking a point in each space and glueing them together. 

\begin{Defn}\label{additive}
We say that a metric space $(X,d)$ is an \emph{additive combination} of metric spaces $(X_1,d_1)$ and $(X_2,d_2)$ if there exist sets $X_1',X_2' \subset X$ and a point $x \in X$ such that:
\begin{enumerate}[(i)]
 \item $X_1' \cup X_2' = X$;
 \item $X_1' \cap X_2' = \seq{x}$;
 \item $(X_1',d)$ is isometrically isomorphic to $(X_1,d_1)$ and $(X_2',d)$ is isometrically isomorphic to $(X_2,d_2)$; and
 \item if $y \in X_1'$ and $z \in X_2'$ then
\[
 d(y,z) = d(y,x) + d(x,z).
\]
\end{enumerate}
 We say that $(X_1,d_1)$ and $(X_2,d_2)$ are \emph{components of $(X,d)$}. The single point in $X_1' \cap X_2'$ will be referred to as the \emph{glue-point} of $(X_1,d_2)$ and $(X_2,d_2)$, and usually be denoted by $x$.
\end{Defn}

We may view additive combinations in two different ways: one is deconstructive, the other constructive. In the deconstructive setting, from a given space we may find two subspaces which can be additively combined to give the original space. Many different decompositions may be possible.

\begin{Exam}\label{G}
Consider the following graph $G$ endowed with the shortest path metric $d$.
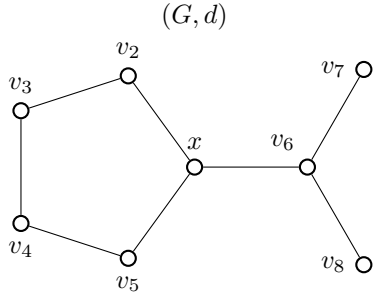
\begin{center}
\begin{tikzpicture}
 \node (G) at (0,2) {$(G,d)$};
 \node[white,label=above:{$x$}] (1) at (0,0) {};
 \path (1) +(0:1.5cm) node[white,label=110:{$v_6$}] (6) {};
 \path (6) +(60:1.5cm) node[white,label=left:{$v_7$}] (7) {};
 \path (6) +(-60:1.5cm) node[white,label=left:{$v_8$}] (8) {};
\draw
(1) -- (6)
(6) -- (7)
(6) -- (8)
;

 \path (1) +(126:1.5cm) node[white,label=above:{$v_2$}] (2) {};
 \path (2) +(-162:1.5cm) node[white,label=above:{$v_3$}] (3) {};
 \path (1) +(-126:1.5cm) node[white,label=below:{$v_5$}] (5) {};
 \path (5) +(162:1.5cm) node[white,label=below:{$v_4$}] (4) {};
\draw
(1) -- (2)
(2) -- (3)
(3) -- (4)
(4) -- (5)
(5) -- (1)
;
\end{tikzpicture}
\end{center}
Then $(G,d)$ can be seen to be the additive combination of the two graphs below, each endowed with the shortest path metric.
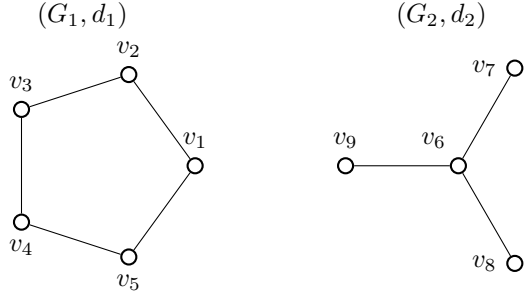
\begin{center}
\begin{tikzpicture}[scale=1]
 \node (G1) at (-1.5,2) {$(G_1,d_1)$};
 \node[white,label=above:{$v_9$}] (9) at (2,0) {};
 \path (9) +(0:1.5cm) node[white,label=110:{$v_6$}] (6) {};
 \path (6) +(60:1.5cm) node[white,label=left:{$v_7$}] (7) {};
 \path (6) +(-60:1.5cm) node[white,label=left:{$v_8$}] (8) {};
\draw
(9) -- (6)
(6) -- (7)
(6) -- (8)
;

 \node (G2) at (3.273,2) {$(G_2,d_2)$};
 \node[white,label=above:{$v_1$}] (1) at (0,0) {};
 \path (1) +(126:1.5cm) node[white,label=above:{$v_2$}] (2) {};
 \path (2) +(-162:1.5cm) node[white,label=above:{$v_3$}] (3) {};
 \path (1) +(-126:1.5cm) node[white,label=below:{$v_5$}] (5) {};
 \path (5) +(162:1.5cm) node[white,label=below:{$v_4$}] (4) {};
\draw
(1) -- (2)
(2) -- (3)
(3) -- (4)
(4) -- (5)
(5) -- (1)
;
\end{tikzpicture}
\end{center}

Note that we could also view $(G,d)$ as the additive combination of the two following graphs, each endowed with the shortest path metric.

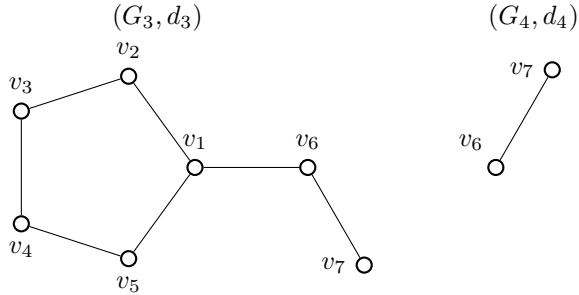
\begin{center}
\begin{tikzpicture}[scale=1]
 \node (G1) at (-0.5,2) {$(G_3,d_3)$};
 \node[white,label=above:{$v_1$}] (1) at (0,0) {};
 \path (1) +(126:1.5cm) node[white,label=above:{$v_2$}] (2) {};
 \path (2) +(-162:1.5cm) node[white,label=above:{$v_3$}] (3) {};
 \path (1) +(-126:1.5cm) node[white,label=below:{$v_5$}] (5) {};
 \path (5) +(162:1.5cm) node[white,label=below:{$v_4$}] (4) {};
 \path (1) +(1.5,0) node[white,label=above:{$v_6$}] (6) {};
 \path (6) +(-60:1.5cm) node[white,label=left:{$v_7$}] (7) {};
\draw
(1) -- (2)
(2) -- (3)
(3) -- (4)
(4) -- (5)
(5) -- (1)
(1) -- (6)
(6) -- (7)
;

 \node (G2) at (4.5,2) {$(G_4,d_4)$};
 \path (2.5,0) +(0:1.5cm) node[white,label=110:{$v_6$}] (6) {};
 \path (6) +(60:1.5cm) node[white,label=left:{$v_7$}] (7) {};
\draw
(6) -- (7)
;
\end{tikzpicture}
\end{center}
\end{Exam}

We can easily extend this idea to more than two spaces. 

\begin{Defn}\label{additivemulti}
 We say that a metric space $(X,d)$ is an additive combination of metric spaces $(X_1,d_1), \ldots, (X_n,d_n)$ if $(X,d)$ may be constructed by successively forming additive combinations of these spaces. That is, there is some ordering $\pi \in S_n$ such that if we first additively combine $(X_{\pi(1)},d_{\pi(1)})$ and $(X_{\pi(2)}, d_{\pi(2)})$, and then additively combine this with $(X_{\pi(3)},d_{\pi(3)})$, and so forth, until all $n$ spaces have been additively combined, the result is $(X,d)$.
\end{Defn}

There may of course be a different ordering $\sigma \in S_n$ that may be used to give the same space. The specific ordering has no effect on the final space, we just require that there be at least one.

We may also view additive combinations constructively: from $n$ spaces $(X_1,d_1),\ldots,(X_n,d_n)$ we may combine them appropriately to form a new space $(X,d)$, which is an additive combination of the $(X_1,d_1),\ldots, (X_n,d_n)$. There may be many non-isomorphic ways of doing this.
\begin{Exam}
 Consider the following three graphs each endowed with the shortest path metric. (We leave them unlabeled for simplicity.)
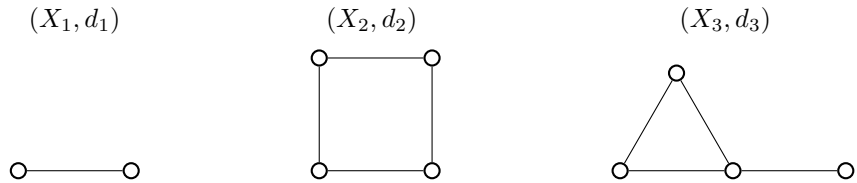
\begin{center}
 \begin{tikzpicture}
 
 \node (X1) at (0.75,2) {$(X_1,d_1)$};
 \node[white] (1) at (0,0) {};
 \path (1) + (1.5,0) node[white] (2) {};  
 \draw (1) -- (2);

 \node (X2) at (4.7,2) {$(X_2,d_2)$};
 \node[white] (3) at (4,0) {};
 \path (3) + (0,1.5) node[white] (4) {};
 \path (4) + (1.5,0) node[white] (5) {};
 \path (5) + (0,-1.5) node[white] (6) {};
 \draw (3) -- (4)
       (4) -- (5)
       (5) -- (6)
       (3) -- (6)
       ;          

 \node (X3) at (9.4,2) {$(X_3,d_3)$};
 \node[white] (7) at (8,0) {};
 \path (7) + (60:1.5cm) node[white] (8) {};
 \path (7) + (0:1.5cm) node[white] (9) {};
 \path (9) + (0:1.5cm) node[white] (10) {};
 \draw (7) -- (8)
       (8) -- (9)
       (7) -- (9)
       (9) -- (10)
       ;

 \end{tikzpicture}
\end{center}
Then there are 16 non-isomorphic graphs which may be formed as additive combinations of $(X_1,d_1),(X_2,d_2)$ and $(X_3,d_3)$. Below are two such examples.
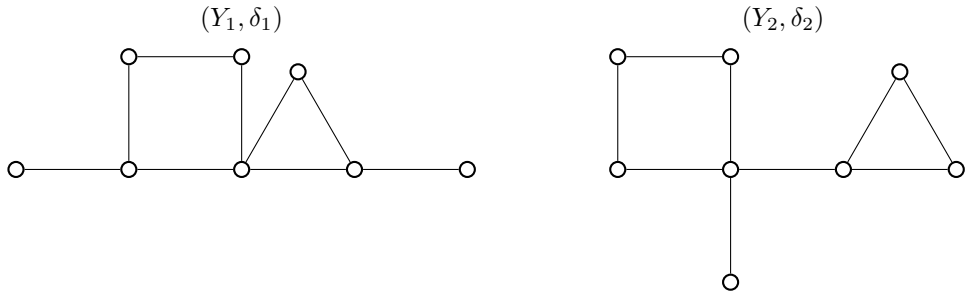
\begin{center}
 \begin{tikzpicture}
\node (Y1) at (3,2) {$(Y_1,\delta_1)$};
\node[white] (1) at (0,0) {};
\path (1) + (1.5,0) node[white] (2) {};
\path (2) + (0,1.5) node[white] (3) {};
\path (3) + (1.5,0) node[white] (4) {};
\path (4) + (0,-1.5) node[white] (5) {};
\path (5) + (0:1.5cm) node[white] (6) {};
\path (5) + (60:1.5cm) node[white] (7) {};
\path (6) + (1.5,0) node[white] (8) {};
\draw (1) -- (2)
      (2) -- (3)
      (3) -- (4)
      (4) -- (5)
      (2) -- (5)
      (5) -- (6)
      (5) -- (7)
      (6) -- (7)
      (6) -- (8)
      ;
\node (Y2) at (10.2,2) {$(Y_2,\delta_2)$};
\node[white] (10) at (8,0) {};
\path (10) + (1.5,0) node[white] (13) {};
\path (10) + (0,1.5) node[white] (11) {};
\path (11) + (1.5,0) node[white] (12) {};
\path (13) + (0,-1.5) node[white] (14) {};
\path (13) + (1.5,0) node[white] (15) {};
\path (15) + (0:1.5cm) node[white] (16) {};
\path (15) + (60:1.5cm) node[white] (17) {};
\draw (14) -- (13) -- (10) -- (11) -- (12) -- (13) -- (15) -- (16) -- (17) -- (15) ;
 \end{tikzpicture}
\end{center}

\end{Exam}

We use the term \emph{additive} to describe this sort of combination since metric spaces that are embeddable in some metric tree $(T,d_T)$ are known as \emph{additive metric spaces}, and we are joining metric spaces together to form `trees of metric spaces'. The fact that we are focusing on trees rather than general graphs is because trees always have a unique path between two distinct vertices, and so the definition of the metric $d$ may be done recursively. A general graph need not have a unique path between two vertices. This would ruin the iterative definition of the metric $d$ in Definition~\ref{additive}.

\section{Generalized roundness $p$}\label{GeneralizedRoundness}

Our proof of Theorem~\ref{main} does not work with $p$-negative type directly, but an equivalent property known as \emph{generalized roundness $p$}. Enflo \cite{Enf69} introduced the ideas of roundness and generalized roundness to answer in the negative a question of Smirnov's: ``Is every separable metric space uniformly homeomorphic to a subset of $L_2[0,1]$?'' In 1997 Lennard, Tonge and Weston \cite{LTW97} showed that the notions of negative type and generalized roundness coincide: a metric space $\bra{X,d}$ has $p$-negative type if and only if it has generalized roundness $p$. The notion of strict generalized roundness $p$ was formalized by Doust and Weston in \cite{DW08a}, and shown to be equivalent to strict $p$-negative type.

Although it is equivalent to $p$ negative type, the setting of generalized roundness $p$ offers a different perspective which we find helpful here. The form in which we will be defining generalized roundness $p$ is slightly non-standard and will require some extra technical definitions, but allow us to prove Theorem~\ref{main} more easily.

\begin{Defn}\label{simplex}
Let $s,t \in \NN$ and $X$ a set. An \emph{$(s,t)$-simplex $D$} is a vector $\bra{a_1,\ldots, a_s, b_1, \ldots, b_t} \in X^{s+t}$ of $(s+t)$ not necessarily distinct points, along with a load vector $\omega = \bra{m_1,\ldots, m_s,n_1,\ldots,n_t} \in \RR^{s+t}_+$ that assigns a non-negative weight $m_j \geq 0$ or $n_i \geq 0$ to each point $a_j$ or $b_i$ respectively, satisfying 
\[
m_1 + \cdots + m_s = n_1 + \cdots + n_t.
\]
We may denote such a simplex $D$ by $[a_i(m_i);b_j(n_j)]$. 
\end{Defn}

The points $a_1, \ldots, a_s$ will be known as the \emph{$a$-team} in $D$, while the $b_1 \ldots, b_t$ will be known as the \emph{$b$-team} in $D$. Note that Definition~\ref{simplex} does not preclude a point $z \in X$ from being a member of both the $a$-team and the $b$-team in a particular simplex. If the number of points in the $a$-team and $b$-team are not immediately relevant, we may refer to a simplex $D$, rather than an $(s,t)$-simplex $D$.

\begin{Defn}\label{GR}
 Let $(X,d)$ be a metric space and $p \geq 0$. Then $(X,d)$ has \emph{generalized roundness $p$} if and only if for all $s,t \in \NN$ and all $(s,t)$-simplices $D = \left[a_i \bra{m_i} ; b_j \bra{n_j} \right]$ in $X$, we have
\begin{equation}\label{roundness}
\gamma^p(D) = \sum_{i,j = 1}^{s,t} m_i n_j d \bra{a_i, b_j}^p - \sum_{1 \leq i_1 < i_2 \leq s} m_{i_1} m_{i_2} d \bra{a_{i_1}, a_{i_2}}^p - \sum_{1 \leq j_1 < j_2 \leq t} n_{j_1} n_{j_2}d \bra{b_{j_1}, b_{j_2}}^p \geq 0.  
\end{equation}
The function $\gamma^p$ is known as the \emph{simplex gap function}.
\end{Defn}

We may write $\gamma$ instead of $\gamma^1$. As we will be working with the simplex gap function extensively, it is convenient to further define the following.

\begin{Defn}\label{LR}
 Let $D=[a_i(m_i);b_j(n_j)]$ be an $(s,t)$-simplex in $(X,d)$. We define for $p \geq 0$ the functions $\fL^p(\cdot)$ and $\fR^p(\cdot)$ by
\[
\fL^p(D) = \sum_{1 \leq i_1 < i_2 \leq s} m_{i_1} m_{i_2} d \bra{a_{i_1}, a_{i_2}}^p + \sum_{1 \leq j_1 < j_2 \leq t} n_{j_1} n_{j_2}d \bra{b_{j_1}, b_{j_2}}^p
\]
and
\[
\fR^p(D) = \sum_{i,j = 1}^{s,t} m_i n_j d \bra{a_i, b_j}^p.
\]
So that
\[
 \gamma^p(D) = \fR^p(D)-\fL^p(D).
\]
\end{Defn}

Before defining strict generalized roundness $p$, we need to deal with the fact that the points in our simplices may not be distinct. This offers us flexibility later, but at a technical cost which we deal with now. In particular, we may have two simplices $D,D'$ that are different with respect to Definition~\ref{simplex}, but for which the sums $\gamma^p(D)$ and $\gamma^p(D')$ are simple re-arrangements of one another. In such a case, we find it convenient to consider such simplices equivalent. To do so we define the following operations.

\begin{Defn}\label{cancel}
Let $D$ be an $(s,t)$-simplex $[a_i(m_i);b_j(n_j)]$ of not necessarily distinct points. We define the following procedures that we may apply to $D$.
\begin{enumerate}[(i)]
 \item Re-index the members of the $a$-team and $b$-team, or swap the roles of all the $a$ and $b$ terms to form a new simplex $D'$.
 \item If $a_1 = a_2$, then form the $(s-1,t)$-simplex 
\[
D' = [ a_1(m_1+m_2),a_3(m_3),\ldots, a_s(m_s) ; b_j(n_j)].
\]
 \item If $a_1 = b_1$ with $m_1 \geq n_1$, then form the $(s,t-1)$-simplex
\[
D' = [a_1(m_1-n_1),a_2(m_2),\ldots, a_s(m_s); b_2(n_2), \ldots, b_t(n_t)].
\]
\item If $m_1=0$ then form the $(s-1,t)$-simplex
\[
 D' = [a_2(m_2),\ldots,a_s(m_s);b_1(n_1), \ldots, b_t(n_t)].
\]
\end{enumerate}
We also allow the inverses of (ii) - (iv), each of which involves adding new points and weights to the simplex.
If a simplex $D''$ may be obtained from $D$ by successively applying the above procedures or their inverses, then we say that $D''$ and $D$ are \emph{equivalent}.
\end{Defn}

Since procedure (i) allows us to re-index the points and swap teams, the procedures (ii) - (iv) and their inverses may be applied to any appropriate points in the simplex, not just the first few of each team. It is not difficult to see that Definition~\ref{cancel} does indeed define an equivalence relation on the collection of weighted simplices in $(X,d)$: reflexivity comes from performing no operations, symmetry comes from performing the reverse of the original procedures, and transitivity from performing two sets of procedures one after the other. The usefulness of the above procedures lies in the following result.

\begin{Lem}
Let $D$ and $D'$ be equivalent weighted simplices as per Definition~\ref{cancel}. Then for all $p \geq 0$ we have $\gamma^p(D) = \gamma^p(D')$.
\end{Lem}

This lemma may be proved by checking that $\gamma^p(D) = \gamma^p(D')$ for any simplices differing by a single application of any of the procedures, or their inverses, in Definition~\ref{cancel}. This is not overly difficult but tedious, coming from directly writing out the sums of both $\gamma^p(D)$ and $\gamma^p(D')$ and matching corresponding terms.

\begin{Defn}
 A simplex $D$ in $(X,d)$ is said to be \emph{degenerate} if it is equivalent to a simplex containing no non-zero weights.
\end{Defn}

In the $p$-negative type setting, a degenerate simplex $D$ corresponds to the null vector $(0, \ldots, 0)$. All non-degenerate simplices correspond to a non-zero vector.

\begin{Defn}
 If $D$ is a non-degenerate simplex, then a \emph{refinement} of $D$ is any simplex $D^*$ that is equivalent to $D$ and has distinct points and strictly positive weights. 
\end{Defn}

The refinements of a non-degenerate simplex $D$ are all related by procedure (i) of definition Definition~\ref{cancel}. That is, we may obtain one from another by simply re-ordering the points and possibly swapping the $a$ and $b$ teams. In this sense, there is essentially a unique refinement for each non-degenerate simplex.

\begin{Defn}\label{weights}
 If $D$ is a non-degenerate simplex, and the $(s,t)$-simplex $D^* = [a_i(m_i) ; b_j (n_j)]$ is a refinement of $D$, then the quantity
\[
 \lambda = \sum_{i=1}^s m_i = \sum_{j=1}^t n_j
\]
is called the \emph{weight} of $D$. We say $D$ is a \emph{$\lambda$-weighted simplex}. If $\lambda =1$ then we say that $D$ is a \emph{normalized simplex}. The weight of a degenerate simplex is defined to be zero.
\end{Defn}

Our above discussion on simplices means we can now define strict generalized roundness.

\begin{Defn}\label{strict}
 Let $(X,d)$ be a metric space and $p \geq 0$. Then $(X,d)$ has \emph{strict generalized roundness $p$} if and only if for all $s,t \in \NN$ and all non-degenerate $(s,t)$-simplices $D = \left[a_i \bra{m_i} ; b_j \bra{n_j} \right]_{s,t}$ in $X$, we have $\gamma^p(D)>0$.
\end{Defn}

In the generalized roundness $p$ setting, the $p$-negative type gap  has a more elegant incarnation. We have:
\begin{equation}\label{typegap2}
 \Gamma_X^p = \inf \seq{ \gamma^p(D) : D \text{ is a normalized simplex in } X }.
\end{equation}
Using a compactness argument, Li and Weston showed in \cite[Theorem 4.1]{LW10} that for finite metric spaces, the infimum in \eqref{typegap2} is actually a minimum. In this case, $(X,d)$ has strict $p$-negative type if and only if $\Gamma^p_X >0$ (see \cite[Theorem~4.1]{LW10}). So if $(X,d)$ is a finite metric space with strict $p$-negative type, then there exists at least one normalized simplex $D$ in $(X,d)$ such that $\gamma^p(D) = \Gamma^p_X$. Such a simplex will be called \emph{extremal}. In the infinite setting we have no such guarantee -- $(X,d)$ may have strict $p$-negative type yet $\Gamma^p_X =0$ (see \cite[Theorem~5.7]{DW08a}). It is in the form \eqref{typegap2} that the $p$-negative type gap was first introduced in \cite{DW08a}. The equivalent form in Definition~\ref{typegap} comes from translating \eqref{typegap2} into the $p$-negative type setting. It is in this translation process that the scaling factor appears in Definition~\ref{typegap}, since the $p$-negative type inequality does not 
require any normalization of $\alpha_1, \ldots, \alpha_n$.

If $D$ is a $\lambda$-weighted simplex, with $\lambda \neq 1$, then we can form a normalized simplex $D'$ by taking a copy of $D$ and dividing all the weights by $\lambda$. Note that this new simplex $D'$ is \emph{not} equivalent to $D$, but has the property that for all $p \geq 0$
\begin{equation}\label{scaling}
 \gamma^p(D') = \frac{1}{\lambda^2} \gamma^p(D).
\end{equation}
Thus we can reformulate \eqref{typegap2} as
\[
  \Gamma_X^p = \inf \seq{ \frac{1}{\lambda^2} \gamma^p(D) : D \text{ is a } \lambda\text{-weighted non-degenerate simplex in } X }.
\]

\section{The $p=1$ Case}

In this section we provide a proof of Theorem~\ref{main} in the case $p=1$. To do this we first establish some lemmas about simplices in additive combinations. 

Since $\Gamma^1_X$ is defined in terms of simplices, we need to move from a single simplex across the whole space to simplices in each component. We first look at how to split a weighted simplex $D$ in an additive connection space $X$ of $X_1$ and $X_2$, into two simplices $D_1$ and $D_2$, one in each component. The basic idea is that $D_1$ and $D$ are the same, except any points and weights in $X_2$ are moved to the joining point $x$. $D_2$ is defined similarly. The details are below.

\begin{Defn}\label{components}
Let $(X,d)$ be an additive combination of $(X_1,d_1)$ and $(X_2,d_2)$ with glue-point $x$. Let $D$ be a non-degenerate simplex in $X$. We define two simplices $D_1,D_2$, the \emph{components of $D$}, in the following way. For $D_1$, start with a copy of $D$. For any point $z \in D$ that belongs to $X_1$, do nothing. For any point $z \in D$ that is an element of $X_2$, substitute the point with $x$, giving $x$ the same weight as the original point, and in the same team. That is
\[
D \rightarrow D_1  \quad : \quad  \begin{cases}
                                   a_i(m_i) \rightarrow a_i(m_i) & \mbox{if } a_i \in X_1\\
                                   b_j(n_j) \rightarrow b_j(n_j) & \mbox{if } b_j \in X_1\\
                                   a_i(m_i) \rightarrow x(m_i) & \mbox{if } a_i \in X_2\\
                                   b_j(n_j) \rightarrow x(n_j) & \mbox{if } b_j \in X_2.
                                   \end{cases}
\]
This process will often mean that the glue-point $x$ belongs to the $a$-team and $b$-team multiple times. A clearly analogous procedure is used to define $D_2$.
\end{Defn}

Note that the above definition allows the possibility that one of $D_1, D_2$ is degenerate. In such a case, the original simplex $D$ is essentially contained in $X_1$ or $X_2$: if $D^*$ is a refinement of $D$, then the points of $D^*$ are either wholly contained in $X_1$ or wholly contained in $X_2$. We may extend the above definition to additive combinations of more than two metric spaces.

\begin{Defn}\label{multiplecomponents}
 Let $(X,d)$ be an additive combination of $(X_1,d_1), \ldots, (X_n,d_n)$. Let $D$ be a non-degenerate simplex in $X$. The \emph{components} of $D$ in $(X,d)$ are the simplices $D_1, \ldots, D_n$ formed in the following way. Let $\pi \in S_n$ be some ordering so that $(X,d)$ may be constructed by additively combining $(X_{\pi(1)},d_{\pi(1)})$ with $(X_{\pi(1)},d_{\pi(2)})$, and then additively combining this with $(X_{\pi(3)},d_{\pi(3)})$ and so forth. Working backwards, split $D$ into two components via Definition~\ref{components}, one for $(X_{\pi(n)},d_{\pi(n)})$, and another for the rest of the space. Continue this process, essentially reversing the construction of $(X,d)$ from the component spaces $(X_1, d_1), \ldots, (X_n,d_n)$. Clearly any other suitable ordering $\pi' \in S_n$ would produce the same components $D_1, \ldots, D_n$, though possibly in a different order.
\end{Defn}

\begin{Exam}\label{simplexexample}
Recall the metric space $(G,d)$ introduced in Example~\ref{G}, the metric combination of spaces $(G_1,d_1)$ and $(G_2,d_2)$. Consider the following normalized simplex in $(G,d)$
\[
 D= [ x(0.3),v_3(0.2),v_4(0.2),v_8(0.3) ; v_6(0.4),v_2(0.5),v_5(0.1) ].
\]
Then Definition~\ref{components} gives that
\[
 D_1 = [x(0.3),v_3(0.2),v_4(0.2),x(0.3);x(0.4),v_2(0.5),v_5(0.1)]
\]
and
\[
 D_2 = [ x(0.3),x(0.2),x(0.2),v_8(0.3);v_6(0.4),x(0.5),x(0.1) ],
\]
which have refinements 
\[
 D_1^* = [ x(0.2),v_3(0.2),v_4(0.2) ; v_2(0.5),v_5(0.1) ] \quad \mbox{and} \quad D_2^* = [x(0.1),v_8(0.3);v_6(0.4)].
\]

\end{Exam}

\begin{Lem}\label{gamma}
Let $(X,d)$ be an additive combination of $(X_1,d_1)$ and $(X_2,d_2)$ with glue-point $x$. Let $D$ be a simplex and let $D_1,D_2$ be the components of $D$. Then
\[
\gamma(D) = \gamma(D_1) + \gamma(D_2)
\]
\end{Lem}

\begin{proof}
It suffices to show that $\fR(D) = \fR(D_1)+\fR(D_2)$ and $\fL(D) = \fL(D_1) + \fL(D_2)$. Relabeling if necessary, suppose that $D$ is such that $a_1, \ldots, a_k \in X_1$, $a_{k+1},\ldots, a_s \in X_2$ and $b_1, \ldots, b_l \in X_1$, $b_{l+1},\ldots, b_t \in X_2$. Then we have
\begin{align*}
\fR(D) & = \sum_{i=1}^s \sum_{j=1}^t m_i n_j d(a_i,b_j)\\
& = \sum_{i=1}^k \sum_{j=1}^l m_i n_j d(a_i,b_j) + \sum_{i=1}^k \sum_{j=l+1}^t m_i n_j d(a_i,b_j) \\
&\quad + \sum_{i=k+1}^s \sum_{j=1}^l m_i n_j d(a_i,b_j) +\sum_{i=k+1}^s \sum_{j=l+1}^t m_i n_j d(a_i,b_j)\\
& =  \sum_{i=1}^k \sum_{j=1}^l m_i n_j d_1(a_i,b_j) + \sum_{i=1}^k \sum_{j=l+1}^t m_i n_j \bra{ d_1(a_i,x) + d_2(x,b_j) } \\
&\quad + \sum_{i=k+1}^s \sum_{j=1}^l m_i n_j \bra{ d_2(a_i,x) + d_1(x,b_j) } +\sum_{i=k+1}^s \sum_{j=l+1}^t m_i n_j d_2(a_i,b_j)\\
& = \sum_{i=1}^k \sum_{j=1}^l m_i n_j d_1(a_i,b_j) + \sum_{i=1}^k \sum_{j=l+1}^t m_i n_j d_1(a_i,x)\\
& \quad +\sum_{i=k+1}^s \sum_{j=1}^l m_i n_j d_1(x,b_j) + \sum_{i=k+1}^s \sum_{j=l+1}^t m_i n_j d_1(x,x)\\
& \quad + \sum_{i=1}^k \sum_{j=1}^l m_i n_j d_2(x,x) + \sum_{i=1}^k \sum_{j=l+1}^t m_i n_j d_2(x,b_j)\\
& \quad + \sum_{i=k+1}^s \sum_{j=1}^l m_i n_j d_2(a_i,x) + \sum_{i=k+1}^s \sum_{j=l+1}^t m_i n_j d_2(a_i,b_j)\\
& = \fR(D_1) + \fR(D_2)
\end{align*}
The proof of $\fL(D) = \fL(D_1) + \fL(D_2)$ is similar, and is omitted.

So we have
\begin{align*}
 \gamma(D) & = \fR(D)-\fL(D)\\
& = \fR(D_1) + \fR(D_2) - \fL(D_1) - \fL(D_2)\\
& = \gamma(D_1) + \gamma(D_2).
\end{align*}

\end{proof}

The component simplices need not be normalized. However, we do have some control over their weights.

\begin{Lem}\label{weightinequality}
Let $(X,d)$ be an additive combination of $(X_1,d_1)$ and $(X_2,d_2)$. Let $D$ be a non-degenerate simplex in $X$ with weight $\lambda$. If $D_1,D_2$ are the components of $D$ with weights $\lambda_1$ and $\lambda_2$ respectively, then
\[
\lambda_1 + \lambda_2 \geq \lambda.
\]
\end{Lem}

\begin{proof}
Since we are interested in the weights of our simplices, it is easiest to work with refined simplices. Let $D^*, D_1^*, D_2^*$ be refinements of $D, D_1, D_2$ respectively. 

Let $x$ be the glue-point of $(X_1,d_1)$ and $(X_2,d_2)$. Suppose $x \neq a_i$ for any $a_i$ in $D^*$. The glue-point $x$ is the only point in $X_1$ whose role in the simplex $D_1^*$ may be different to its role in $D^*$. So we have
\[
 \lambda_1 \geq \sum_{i : a_i \in X_1} m_i.
\]
Similarly, none of the points in $X_2$ that belong to the $a$-team have their weights diminished when forming $D_2^*$, so
\[
 \lambda_2 \geq \sum_{i: a_i \in X_2} m_i.
\]
As $x \neq a_i$ for any $i$, we have covered all of the members of the $a$-team, and so
\[
 \lambda_1 + \lambda_2 \geq \sum_{i:a_i \in X_1} m_i + \sum_{i:a_i \in X_2} m_i = \sum_{i:a_i \in X} m_i = \lambda.
\]
An analogous argument shows that $\lambda_1 + \lambda_2 \geq \lambda$ if $x \neq b_j$ for any $b_j$ in $D^*$. As we are working with a refined simplex $D^*$, this covers all possible cases. 
\end{proof}

Note that the above Lemma cannot be strengthened to $\lambda_1 + \lambda_2 = \lambda$, as shown by the following example.

\begin{Exam}
 Recall again the space $(G,d)$, and consider the normalized simplex 
\[
D=[v_5(0.4),v_7(0.6);v_2(0.6),v_8(0.4)].
\]
Then the refined component simplices are seen to be
\[
 D_1^* = [x(0.2),v_5(0.4);v_2(0.2)] \quad \text{and} \quad D_2^* = [v_7(0.6);x(0.2),v_8(0.4)],
\]
with weights $\lambda_1 = 0.6$ and $\lambda_2 =0.6$ respectively. So we have $\lambda_1 + \lambda_2 = 1.2 > 1 = \lambda$.
\end{Exam}

A straightforward inductive argument gives the following.

\begin{Cor}\label{multiweights}
 Let $(X,d)$ be an additive combination of $(X_1,d_1), \ldots, (X_n,d_n)$. Let $D$ be a non-degenerate simplex in $X$ with weight $\lambda$. If $D_1, \ldots, D_n$ are the components of $D$ with weights $\lambda_1, \ldots, \lambda_n$ respectively, then
 \[
  \lambda_1 + \cdots + \lambda_n \geq \lambda.
 \]
\end{Cor}

We now have enough information to prove the $p=1$ case of Theorem~\ref{main}.

\begin{Thm}\label{pequals1case}
Let $(X,d)$ be an additive combination of $(X_1,d_1),\ldots, (X_n,d_n)$. 
\begin{enumerate}[(i)]
 \item If $(X_1,d_1),\ldots,(X_n,d_n)$ all have $1$-negative type, then so does $(X,d)$.
 \item If $(X_1,d_1),\ldots,(X_n,d_n)$ all have strict $1$-negative type, then so does $(X,d)$.
 \item If $\Gamma^1_{X_1},\ldots,\Gamma^1_{X_n} > 0$, then $\Gamma^1_X >0$ and is given by
\[
\Gamma^1_X = \bra{ \sum_{i=1}^n \bra{\Gamma^1_{X_i}}^{-1} }^{-1}.
\]
\end{enumerate}
\end{Thm}

\begin{proof}
Suppose $(X,d)$ is an additive combination of $(X_1, d_1), \ldots, (X_n,d_n)$. 

Parts (i) and (ii) follow from Lemma~\ref{gamma}. An inductive argument gives
\[
 \gamma(D) = \sum_{i=1}^n \gamma(D_i)
\]
for any normalized simplex $D$ in $X$. If $(X_1,d_2), \ldots, (X_n,d_n)$ all have 1-negative type, then $\gamma(D_1),\ldots,\gamma(D_n)\geq0$. So $\gamma(D)\geq 0$, and we conclude that $(X,d)$ also has 1-negative type, proving (i). If $(X_1,d_1), \ldots, (X_n,d_n)$ all have strict 1-negative type, and $D$ is a normalized simplex in $(X,d)$ then by Corollary~\ref{multiweights} at least one of the components $D_1,\ldots, D_n$, say $D_k$, has non-zero weight and so is non-degenerate. Since $(X_k,d_k)$ has strict 1-negative type, we have $\gamma(D_k) >0$. Thus
\[
 \gamma(D) = \sum_{i=1}^n \gamma(D_i) \geq \gamma(D_k) >0,
\]
which shows that $(X,d)$ also has strict 1-negative type, proving (ii).

For (iii) the proof also proceeds via induction. The base case of joining two spaces $(X_1,d_1)$ and $(X_2,d_2)$ takes some work. The inductive step is essentially the base case again, and so does not require much more work.

Let $(X,d)$ be an additive combination of $(X_1,d_1)$ and $(X_2,d_2)$. Let $D$ be a normalized simplex in $(X,d)$ and $D_1, D_2$ be its components, with corresponding weights $\lambda_1, \lambda_2$. Note by Lemma~\ref{weightinequality} we have $\lambda_1+\lambda_2 \geq 1$. By Lemma~\ref{gamma} we have
\[
 \gamma(D) = \gamma(D_1) + \gamma(D_2).
\]
The component simplices $D_1$ and $D_2$ are not necessarily normalized. Let $D_1'$ and $D_2'$ be the normalized versions of $D_1$ and $D_2$ respectively. That is, the same points but with the weights scaled so as to be normalized but keeping the same ratios of weights. As $D_1$ and $D_2$ have weights $\lambda_1$, $\lambda_2$ respectively, this means that all the weights in $D_1$ are simply those in $D_1'$ multiplied by $\lambda_1$, and all the weights in $D_2$ are simply those in $D_2'$ multiplied by $\lambda_2$. We have
\begin{align*}
 \Gamma^1_X & = \inf \seq{ \gamma(D) : D \text{ is a normalized simplex in } X}\\
& = \inf \seq{ \gamma(D_1) + \gamma(D_2) : D \text{ is a normalized simplex in } X}\\
& = \inf \seq{ \lambda_1^2 \gamma(D_1') + \lambda_2^2 \gamma(D_2') : D \text{ is a normalized simplex in } X} \tag{by \eqref{scaling}}\\
& \geq \inf \seq{ \lambda_1^2 \gamma(E) + \lambda_2^2 \gamma(F) : E, F \text{ are normalized simplices in } X_1,X_2 \text{ and } \lambda_1 + \lambda_2 \geq 1}\\
& = \inf \seq{ \lambda_1^2 \Gamma^1_{X_1} + \lambda_2^2 \Gamma^1_{X_2} : \lambda_1 + \lambda_2 \geq 1 }\\
& = \inf \seq{ \lambda_1^2 \Gamma^1_{X_1} + \lambda_2^2 \Gamma^1_{X_2} : \lambda_1 + \lambda_2 = 1 }.
\end{align*}
This last infimum can be computed directly. We see that it is actually a minimum, with value
\begin{equation}\label{minimum}
\bra{\bra{\Gamma^1_{X_1}}^{-1} + \bra{ \Gamma^1_{X_2} }^{-1} }^{-1},
\end{equation}
which occurs when
\[
\lambda_1 = \frac{\Gamma^1_{X_2}}{ \Gamma^1_{X_1} + \Gamma^1_{X_2} } \quad \mbox{and} \quad \lambda_2 = \frac{\Gamma^1_{X_1}}{ \Gamma^1_{X_1} + \Gamma^1_{X_2} }
\]
Thus
\[
 \Gamma^1_{X} \geq \bra{\bra{\Gamma^1_{X_1}}^{-1} + \bra{ \Gamma^1_{X_2} }^{-1} }^{-1}.
\]

Next we show that the value \eqref{minimum} is also an upper bound for $\Gamma^1_X$. If $(X,d)$ has finitely many points, then we may produce a normalized simplex $D$ in $X$ such that $\gamma(D)$ is equal to the expression in \eqref{minimum}. The case of $(X,d)$ having infinitely many points may be dealt with via the finite case and a suitable limiting argument. 

Suppose that $(X,d)$ is a finite metric space. Then $(X_1,d_1)$ and $(X_2,d_2)$ are also finite metric spaces. So there exist normalized simplices $D_1'$ in $X_1$ and $D_2'$ in $X_2$ that are extremal in that
\[
 \gamma(D_1') = \Gamma^1_{X_1} \quad \mbox{and} \quad \gamma(D_2')= \Gamma^1_{X_2}.
\]
Without loss of generality, we also assume that the glue-point $x$ does not belong to the $b$-team of either $D_1$ or $D_2$. Define
\[
 \lambda_1 = \frac{\Gamma^1_{X_2}}{\Gamma^1_{X_1} + \Gamma^1_{X_2}} \quad \mbox{and} \quad \lambda_2 = \frac{\Gamma^1_{X_1}}{\Gamma^1_{X_1} + \Gamma^1_{X_2}},
\]
so that $\lambda_1 + \lambda_2 = 1$.

Now, let $D_1$ and $D_2$ denote the weighted simplices formed by multiplying the weights of $D_1'$ by $\lambda_1$ and $D_2'$ by $\lambda_2$. We now construct a normalized simplex $D$ in $X$ such that its components are the $D_1$ and $D_2$ just defined. Construct $D$ as follows:
\begin{enumerate}[(i)]
 \item Firstly, for all points $z \in X - \seq{x}$, use $D_1$ or $D_2$ to determine if $z$ belongs to the $a$-team or the $b$-team, or neither team, and also its weighting. Note that as the points in $D_1$ and $D_2$ are distinct (apart from possibly at $x$), this is well defined.
 \item Secondly, we deal with the glue-point $x$. Let the $a$-team weight of $x$ in $D_1$ be denoted by $m_{D_1}(x)$, with this quantity equal to 0 if $x$ is not a member of the $a$-team in $D_1$. Similarly define $m_{D_2}(x)$. If $m_{D_1}(x) + m_{D_2}(x) >0$, then let $x$ be a member of the $a$-team in $D$ with weight $m_{D}(x) =m_{D_1}(x) + m_{D_2}(x)$. If $m_{D_1}(x) + m_{D_2}(x) =0$, then simply omit $x$ from the simplex $D$.
\end{enumerate}
We claim that the above produces a normalized simplex $D$ in $X$. Indeed, we have
\begin{align*}
 \sum_{j: b_j \in X} n_j & = \sum_{j:b_j \in X_1} n_j + \sum_{j:b_j \in X_2} n_j\\
                      & = \lambda_1 + \lambda_2\\
                      & =1, 
\end{align*}
and
\begin{align*}
 \sum_{i:a_i \in X} m_i & = m_D(x) + \sum_{i: a_i \in X_1 -\seq{x} } m_i + \sum_{ i: a_i \in X_2 -\seq{x}  } m_i\\
& = m_{D_1(x)} + m_{D_2}(x) + \sum_{i: a_i \in X_1 -\seq{x}} m_i + \sum_{i: a_i \in X_2 -\seq{x}} m_i\\
& = \sum_{i: a_i \in X_1} m_i + \sum_{i:a_i \in X_2} m_i\\
& = \lambda_1 + \lambda_2\\
& = 1.
\end{align*}
So $D$ is indeed a normalized simplex in $X$. 

Finally, we see that the constructed normalized simplex $D$ has all the desired properties. By construction, the components of $D$ via Definition~\ref{components} are exactly $D_1$ and $D_2$ defined above, with weights $\lambda_1$ and $\lambda_2$. By Lemma~\ref{gamma} we have
\begin{align*}
 \gamma(D) &= \gamma(D_1) + \gamma(D_2)\\
           &= \lambda_1^2 \gamma(D_1') + \lambda_2^2 \gamma(D_2')\\
           &= \bra{\frac{\Gamma^1_{X_2}}{\Gamma^1_{X_1} + \Gamma^1_{X_2}}}^2 \Gamma^1_{X_1} +\bra{\frac{\Gamma^1_{X_1}}{\Gamma^1_{X_1} + \Gamma^1_{X_2}}}^2 \Gamma^1_{X_2}\\
           &= \frac{ \Gamma^1_{X_1} \Gamma^1_{X_2} }{\Gamma^1_{X_1} + \Gamma^1_{X_2}}\\
           &= \bra{ \bra{\Gamma^1_{X_1}}^{-1} + \bra{\Gamma^1_{X_2}}^{-1} }^{-1}.
\end{align*}
As $\Gamma^1_X \leq \gamma(D)$, this completes the finite case. 

If $(X,d)$ is an infinite metric space, then extremal simplices $D_1'$ and $D_2'$ do not necessarily exist. However, by the definitions of $\Gamma^1_{X_1}$ and $\Gamma^1_{X_2}$, for any $\epsilon >0$ there exist normalized finite simplices $D_1'(\epsilon)$ and $D_2'(\epsilon)$ such that $\gamma \bra{ D_1'(\epsilon) } = \Gamma^1_{X_1} + \epsilon$ and $\gamma \bra{ D_2'(\epsilon) } = \Gamma^1_{X_2} + \epsilon$. Going through the above procedure gives a normalized simplex $D(\epsilon)$ such that
\[
\gamma(D(\epsilon)) =  \bra{ \bra{\Gamma^1_{X_1}+\epsilon}^{-1} + \bra{\Gamma^1_{X_2}+\epsilon}^{-1} }^{-1}.
\]
As $\Gamma^1_X \leq \gamma(D(\epsilon))$, taking $\epsilon$ to $0$ gives the result. This concludes the proof of the base case for our induction.

Now suppose $(X,d)$ is an additive combination of $(X_1,d_1),\ldots,(X_{k+1},d_{k+1})$. Then $(X,d)$ can be formed by joining the $k+1$ spaces successively, each time forming an additive combination. So, relabeling if necessary, we can consider $(X,d)$ as an additive combination of $(Y,\delta)$ and $(X_{k+1},d_{k+1})$, where $(Y,\delta)$ is an an additive combination of $(X_1,d_1),\ldots,(X_k,d_k)$. But this is simply an additive combination of two spaces, and so by the base case
\begin{equation}\label{basecase}
\Gamma^1_X = \bra{ \bra{ \Gamma^1_Y}^{-1} + \bra{ \Gamma^1_{X_{k+1}} }^{-1} }^{-1}.
\end{equation}
But by the inductive hypothesis
\[
\Gamma^1_{Y} = \bra{ \sum_{i=1}^k \bra{\Gamma^1_{X_i}}^{-1} }^{-1}.
\]
Hence equation~\eqref{basecase} simplifies to
\[
\Gamma^1_X = \bra{ \sum_{i=1}^{k+1} \bra{\Gamma^1_{X_i}}^{-1} }^{-1}.
\]
So by mathematical induction we are done. Thus the theorem is proved for the case $p=1$.
\end{proof}

\begin{Rem}
Theorem~\ref{pequals1case} extends some previous work in several directions. In particular, the result in \cite{HLMT98} that all finite metric trees have strict 1-negative type follows from Theorem~\ref{main} part (ii), as all metric trees can be thought of additive combinations of their edges, each having strict 1-negative type. In \cite{DW08a}, Doust and Weston extended some of the work done in \cite{HLMT98} by finding an alternative proof that all finite metric trees have strict 1-negative type, and calculating the 1-negative type gap for finite weighted metric trees as
\begin{equation}\label{treeformula}
  \Gamma^1_T = \seq{ \sum_{ e \in E(T) } \abs{e}^{-1} }^{-1},
\end{equation}
where $E(T)$ denotes the set of edges in $T$ and $\abs{e}$ the length of edge $e$. This follows directly from Theorem~\ref{main} part (iii), by noting that each finite weighted metric tree can be formed as the additive combination of its edges, each two-point metric spaces. Each edge $e$ has $\Gamma^1_e = \abs{e}$, and so the formula \eqref{treeformula} can be seen as a special case of part (iii).
 
\end{Rem}

\section{The General Case}

Now that the $p=1$ case has been established in the previous section, we are able to extend to the full proof of Theorem~\ref{main} without much more work. We first recall facts about scaling metric spaces.

\begin{Defn}
 If $(X,d)$ is a metric space and $c>0$, then $(X,d^c)$ is the (semi)-metric space on the set $X$ with distance defined by $d^c(x,y) = d(x,y)^c$ for $x,y \in X$. We may use the abbreviation $X^c$ for $(X,d^c)$. \end{Defn}

Note that if $(X,d)$ is a metric space and $0 < c \leq 1$, then $X^c$ is also a metric space. If $c > 1$ then the triangle inequality may fail to hold in $X^c$, in which case $X^c$ is only a semi-metric space. The definitions of $p$-negative and generalized roundness $p$ extend naturally to semi-metric spaces, since the triangle inequality is not used in any way. From now on we will not distinguish between metric and semi-metric spaces, simply referring to a ``space $(X,d)$''.
 
The $p$-negative type properties of $X^c$ follow directly from the $p$-negative type properties of $X$. Indeed, we can see that if $(X,d)$ has (strict) $q$-negative type, then $(X,d^c)$ has (strict) $\bra{\frac{q}{c}}$-negative type. It therefore follows that if $(X,d)$ has finite supremal $p$-negative type and $c>0$, then
\[
 \wp(X,d^c) = \frac{1}{c} \wp(X,d).
\]
Additionally, we can easily see by referring to Definition~\ref{typegap}, that for $q,c >0$ we have
\begin{equation}\label{typegapidentity}
 \Gamma^q_{X^c} = \Gamma^{qc}_{X} = \Gamma^1_{X^{qc}},
\end{equation}
as both $q$ and $c$ appear in the exponent of $d$.
 
With these facts about scaled metric spaces, we may define $p$-additive combinations.
 
\begin{Defn}\label{padditive}
Let $p>0$. We say that a space $(X,d)$ is a \emph{ $p$-additive combination } of spaces $(X_1,d_1), \ldots, (X_n,d_n)$ if we may view $(X,d^p)$ as an additive combination of $(X_1,d^p_1), \ldots, (X_n,d_n^p)$. 
\end{Defn}
 
Intuitively, we form $(X,d)$ from the $(X_1, d_1), \ldots (X_n,d_n)$ by first deforming the component spaces by raising their metric to the exponent $p$, then additively combining them to form $(X,d^p)$, and then obtain $(X,d)$ by deforming again, this time by raising the new metric to the exponent $1 /p$. This deformation and then reverse deformation means that the spaces $(X_1,d_1),\ldots, (X_n,d_n)$ are all isometrically embedded in $(X,d)$, so we can genuinely think of them as pieces of $(X,d)$. The only difference between this and an additive combination of $(X_1,d_1), \ldots, (X_n,d_n)$ is how we join the metrics together. It is also clear that a $1$-additive combination is the same as an additive combination, so Theorem~\ref{pequals1case} is truly a subcase of Theorem~\ref{main}.

With the above definitions we are able to complete our proof of Theorem~\ref{main}.

\begin{proof}[Proof of Theorem~\ref{main}]
 Let $(X,d)$ be a $p$-additive combination of $(X_1,d_1), \ldots, (X_n,d_n)$. Then by Definition~\ref{padditive}, $X^p$ is an additive combination of $X^p_1, \ldots, X^p_n$. If $X_1, \ldots, X_n$ all have $p$-negative type, then $X^p_1, \ldots, X^p_n$ all have 1-negative type. By Theorem~\ref{pequals1case} part (i) we conclude that $X^p$ also has 1-negative type, so $X$ has $p$-negative type. This gives part (i). A clearly similar argument also gives the strict $p$-negative type case, giving part (ii).
 
 For part (iii), we note that if $\Gamma^p_{X_1}, \ldots, \Gamma^p_{X_n} > 0$ then by \eqref{typegapidentity} we have $\Gamma^1_{X^p_1}, \ldots, \Gamma^1_{X^p_1} > 0$. Since $X^p$ is an additive combination of $X^p_1, \ldots, X^p_n$, we conclude by Theorem~\ref{pequals1case} part (iii) that
 \[
  \Gamma^1_{X^p} = \bra{ \sum_{i=1}^n  \bra{ \Gamma^1_{X^p_i} }^{-1}  }^{-1}.
 \]
Using \eqref{typegapidentity} again, we see that this is the same as
\[
   \Gamma^p_{X} = \bra{ \sum_{i=1}^n  \bra{ \Gamma^p_{X_i} }^{-1}  }^{-1},
\]
as required.
\end{proof}

We now give an example of Theorem~\ref{main} in action, on a relatively small space. 
\begin{Exam}\label{examplelower}
 Consider again $(G,d)$ from Example~\ref{G}, an additive combination of $(G_1,d_1)$ and $(G_2,d_2).$
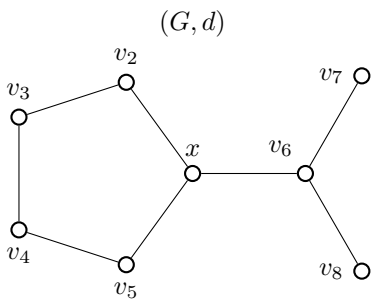
\begin{center}
\begin{tikzpicture}[scale=1]
 \node (G) at (0,2) {$(G,d)$};
 \node[white,label=above:{$x$}] (1) at (0,0) {};
 \path (1) +(0:1.5cm) node[white,label=110:{$v_6$}] (6) {};
 \path (6) +(60:1.5cm) node[white,label=left:{$v_7$}] (7) {};
 \path (6) +(-60:1.5cm) node[white,label=left:{$v_8$}] (8) {};
\draw
(1) -- (6)
(6) -- (7)
(6) -- (8)
;

 \path (1) +(126:1.5cm) node[white,label=above:{$v_2$}] (2) {};
 \path (2) +(-162:1.5cm) node[white,label=above:{$v_3$}] (3) {};
 \path (1) +(-126:1.5cm) node[white,label=below:{$v_5$}] (5) {};
 \path (5) +(162:1.5cm) node[white,label=below:{$v_4$}] (4) {};
\draw
(1) -- (2)
(2) -- (3)
(3) -- (4)
(4) -- (5)
(5) -- (1)
;
\end{tikzpicture}
\end{center}
The 1-negative type gap of $G_1$ may be calculated by \cite[Theorem~3.5]{Wol12} to be
\[
 \Gamma^1_{G_1} = \frac{5}{28}, 
\]
which can be attained by the extremal simplex (found via basic calculus)
\[
 D_{G_1} = \left[ x \bra{ \frac{4}{7} } , v_3 \bra{ \frac{3}{14} }, v_4 \bra{\frac{3}{14}} ; v_2 \bra{ \frac{1}{2} }, v_5 \bra{ \frac{1}{2} } \right].
\]
The 1-negative type gap of $G_2$ can be calculated using Theorem~\ref{pequals1case} part (iii), as it is the additive combination of three edges. Each edge has 1-negative type gap equal to 1, so we have
\[
 \Gamma^1_{G_2} = \bra{1^{-1} + 1^{-1} + 1^{-1}}^{-1} = \frac{1}{3},
\]
which is attained by the extremal simplex
\[
D_{G_2} = \left[ x \bra{ \frac{1}{3}}, v_7 \bra{ \frac{1}{3}}, v_8 \bra{ \frac{1}{3} }   ; v_6 \bra{1} \right].
\]

So by the above theorem we have
\begin{align*}
 \Gamma_G^1 & = \bra{\bra{ \frac{5}{28} }^{-1} + \bra{ \frac{1}{3} }^{-1} }^{-1}\\
            & = \frac{5}{43}. 
\end{align*}

Next, we construct for $(G,d)$ such an extremal simplex $D$ as described in the above theorem. Following the procedure in the proof of Theorem~\ref{pequals1case}, we set
\[
 \lambda_1 = \frac{\Gamma^1_{G_2}}{\Gamma^1_{G_1} + \Gamma^1_{G_2}} = \frac{28}{43}
\]
and
\[
 \lambda_2 = \frac{\Gamma^1_{G_1}}{\Gamma^1_{G_1} + \Gamma^1_{G_2}} = \frac{15}{43}.
\]
We already have examples of extremal simplices for $(G_1,d_1)$ and $(G_2,d_2)$, so we may set
\[
 D_1' = D_{G_1} \quad \text{and} \quad D_2' = D_{G_2}.
\]
Thus using the weights $\lambda_1$ and $\lambda_2$ we have
\[
 D_1 = \left[ x \bra{ \frac{16}{43} } , v_3 \bra{ \frac{6}{43} }, v_4\bra{\frac{6}{43}} ; v_2 \bra{ \frac{14}{43} }, v_5 \bra{ \frac{14}{43} } \right].
\]
and
\[
 D_2  = \left[ x \bra{ \frac{5}{43}}, v_7 \bra{ \frac{5}{43}}, v_8 \bra{ \frac{5}{43} }   ; v_6 \bra{ \frac{15}{43}} \right].
\]
We need to determine the weighting of $x$ in $D$. Its weight in $D_1$ is $\frac{16}{43} = m_{D_1}(x)$, while its weight in $D_2$ is $\frac{5}{43}= m_{D_2}(x)$ so we have
\begin{align*}
 \lambda_D(x) &= m_{D_1}(x) + m_{D_2}(x)\\
              &= \frac{16}{43} + \frac{5}{43}\\
              &=\frac{21}{43}.
\end{align*}
All the other weights come straight from $D_1$ and $D_2$, so
\[
 D = \left[ x \bra{\frac{21}{43}},v_3\bra{\frac{6}{43}},v_4\bra{\frac{6}{43}},v_7\bra{\frac{5}{43}},v_8\bra{\frac{5}{43}} ; v_2\bra{\frac{14}{43}},v_5\bra{\frac{14}{43}},v_6\bra{\frac{15}{43}} \right].
\]
Note that $D$ is a normalized simplex. We can calculate in straight-forward manner and see that indeed
\[
 \gamma(D) = \frac{5}{43}.
\]
\end{Exam}

\section{A Lower Bound Application}\label{boundsection}

The formula for the $p$-negative type gap of $p$-additive combination spaces can be used to provide a lower-bound on the supremal $p$-negative type of such spaces. This comes from combining work in \cite[Theorem~3.3]{LW10} with Theorem~\ref{main} part (iii). First we require some more notation.

\begin{Defn}
 For $n \geq 2$ let
\[
 c(n) = 1 - \frac{1}{2} \bra{ \frac{1}{\lfloor \frac{n}{2} \rfloor} + \frac{1}{\lceil \frac{n}{2 } \rceil} }.
\]
\end{Defn}

\begin{Defn}
 Let $(X,d)$ be a finite metric space. The \emph{scaled diameter of $(X,d)$} is
\[
 \mathfrak D_X = \frac{ \diam(X) }{ \min \seq{ d(x_1,x_2) : x_1 \neq x_2 } }.
\]
\end{Defn}

Recall the following from Li-Weston~\cite{LW10}, slightly paraphrased.

\begin{Thm}\label{LiWeston}
Let $(X,d)$ be a finite metric space with cardinality $n = \abs{X} \geq 3$ and let $p \geq 0$. If the $p$-negative type gap $\Gamma^p_X$ of $(X,d)$ is positive, then
\[
 \wp(X,d) \geq p + \frac{ \ln \bra{ 1 + \frac{\Gamma^p_X}{ \mathfrak D_X^p \cdot c(n) } } }{ \ln \mathfrak D_X }.
\]
\end{Thm}

Note that Theorem~\ref{LiWeston} requires that the $\Gamma^p_X$ comes from the (possibly rescaled) version of $(X,d)$ in which $\mathfrak{D}_X = \diam(X)$, although this not explicitly stated in \cite{LW10}.

From Theorem~\ref{LiWeston} and Theorem~\ref{main} part (iii) we obtain the following.

\begin{Theorem}\label{boundimproved}
 Let $(X_1,d_1),\ldots, (X_n,d_n)$ be finite spaces with minimum non-zero distances equal to 1. If $(X,d)$ is any $p$-additive combination of $(X_1,d_1),\ldots,(X_n,d_n)$ and each of the $(X_1,d_1),\ldots,(X_n,d_n)$ has strict $p$-negative type, then
\[
 \wp(X,d) \geq p + p \frac{ \ln \bra{ 1 + \bra{\sum_{i=1}^n \bra{\Gamma^p_{X_i}}^{-1}}^{-1} \cdot \bra{ \sum_{i=1}^n \diam(X_i)^p }^{-\frac{1}{p}} \cdot c \bra{ \sum_{i=1}^n \abs{X_i} -n +1  }^{-1} } }{ \ln \bra{ \sum_{i=1}^n \diam(X_i)^p } }.
\]

\end{Theorem}

\begin{proof}

Theorem~\ref{main} gives us an expression for $\Gamma^p_X$ in terms of $\Gamma^p_{X_1}, \ldots, \Gamma^p_{X_n}$. So the only other terms we need to consider are $\abs{X}$ and $\mathfrak D_X$.

By the definition of $p$-additive combination spaces, if $(X,d)$ is the $p$-additive combination of $(X_1,d_1)$ and $(X_2,d_2)$, then $\abs{X} = \abs{X_1} + \abs{X_2} -1$, provided all are finite. Therefore, for our space $(X,d)$ which is the $p$-additive combination of $n$ spaces, we have
\[
 \abs{X} = \sum_{i=1}^n \abs{X_i} -n + 1.
\]

Since the minimum non-zero distance in all of $(X_1,d_1),\ldots,(X_n,d_n)$ is 1, we have $\mathfrak D_{X_i} = \diam(X_i)$ for $i=1, \ldots, n$. Considering all possible ways to combine $(X_1,d_1), \ldots, (X_n,d_n)$ to form $(X,d)$, we see that 
\[
 \min_{i=1,\ldots, n} \seq{ \diam(X_i)} \leq \diam(X) \leq \bra{ \sum_{i=1}^n \diam(X_i)^p }^{\frac{1}{p}}.
\]
Using this upper bound for $\diam(X)$, $\Gamma^p_{X}$ from Theorem~\ref{main} part (iii) and our above formula for $\abs{X}$, combined with Theorem~\ref{LiWeston} we have
\begin{align*}
 \wp(X,d) & \geq p + \frac{ \ln \bra{ 1 + \frac{\Gamma^p_X}{ \mathfrak D_X^p \cdot c(\abs{X}) } } }{ \ln \mathfrak D_X } \\
          & = p + \frac{ \ln \bra{ 1 + \frac{\Gamma^p_X}{ \diam(X)^p \cdot c(\abs{X}) } } }{ \ln \bra{\diam(X)} }\\
          & \geq p + \frac{ \ln \bra{ 1 + \bra{\sum_{i=1}^n \bra{\Gamma^p_{X_i}}^{-1}}^{-1} \cdot \bra{ \sum_{i=1}^n \diam(X_i)^p }^{-\frac{1}{p}} \cdot c \bra{ \sum_{i=1}^n \abs{X_i} -n +1  }^{-1} } }{ \ln \bra{ \bra{ \sum_{i=1}^n \diam(X_i)^p  }^\frac{1}{p} } }\\
          & = p + p \frac{ \ln \bra{ 1 + \bra{\sum_{i=1}^n \bra{\Gamma^p_{X_i}}^{-1}}^{-1} \cdot \bra{ \sum_{i=1}^n \diam(X_i)^p }^{-\frac{1}{p}} \cdot c \bra{ \sum_{i=1}^n \abs{X_i} -n +1  }^{-1} } }{ \ln \bra{ \sum_{i=1}^n \diam(X_i)^p } }.
\end{align*}

\end{proof}

Note that we cannot do away with the assumption that the minimum non-zero distances in $(X_1d_1), \ldots, (X_n,d_n)$ are all 1. In general, nothing can be said about $\mathfrak D_X$ in terms of the $\mathfrak D_{X_1},\ldots, \mathfrak D_{X_n}$ alone, as shown by the following example.

\begin{Exam}
Let $X_1 = \seq{a,b}$ with $d(a,b) =1$, and let $X_2 = \seq{b,c}$ with $d(b,c) = \alpha \geq 1$. Now let $(X,d)$ be the additive combination of $(X_1,d_1)$ and $(X_2,d_2)$ formed by joining the spaces together at $b$. Then $\mathfrak D_{X_1} = \mathfrak D_{X_2} =1$, but
\[
 \mathfrak D_{X} = \alpha + 1
\]
As $\alpha$ may be any positive number, $\mathfrak D_X$ can take any value in the interval $[1,\infty)$, even though $\mathfrak D_{X_1}= \mathfrak D_{X_2} =1$.
\end{Exam}

\begin{Exam}
 In the example $(G,d)$ from before we found that $\Gamma^1_G = \frac{5}{43}$. We have $\abs{G} = 8$, and the scaled diameter is $\mathfrak D_G = 4$. So by Theorem~\ref{LiWeston} bound we have
\[
 \wp(G,d) \geq 1 + \frac{ \ln \bra{ 1 + \frac{\frac{5}{43}}{ 4 \cdot \frac{3}{4} } } }{ \ln 4 } = 1.027...
\]
Using \cite[Corollary~2.4]{San12} can calculate approximately that $\wp(G,d) = 1.36..$ So the bound is not at all sharp. This is to be expected since the lower-bound is uniform for many quite different spaces, which one expects (and finds experimentally) to have quite different supremal $p$-negative types. 
\end{Exam}

\bibliographystyle{alpha}
\bibliography{/home/nfs/z3206935/hdrive/Bibliography/bibliography.bib}

\end{document}